\documentclass[11pt]{amsart}
\usepackage{epsfig,amsfonts,amsthm,amssymb,latexsym,amsmath}

\newcommand{\mymod}[3]{#1 \equiv #2 \kern -0.5em \pmod{#3}}
\newcommand{\mynotmod}[3]{#1 \not \equiv #2 \kern -0.6em \pmod{#3}}

\theoremstyle{plain}
\newtheorem{theorem}{Theorem}[section]
\newtheorem{corollary}[theorem]{Corollary}
\newtheorem{lemma}[theorem]{Lemma}

\theoremstyle{remark}

\theoremstyle{definition}
\newtheorem{definition}[theorem]{Definition}

\title{Third-order Jacobsthal $3$-Parameter Generalized Quaternions}
\vspace{5pt}

\author{\scriptsize GAMALIEL MORALES}
\date{}

\begin{document}
\maketitle

\vspace{-20pt}
\begin{center}
{\footnotesize Instituto de Matem\'aticas,\\
Pontificia Universidad Cat\'olica de Valpara\'iso,\\
Blanco Viel 596, Valpara\'iso, Chile.\\
E-mail: gamaliel.cerda.m@mail.pucv.cl
}\end{center}

\vspace{20pt}
\hrule

\begin{abstract}
The purpose of this article is to bring together the third-order Jacobsthal numbers and 3-parameter generalized quaternions, which are a general form of the quaternion algebra according to 3-parameters. With this purpose, we introduce and examine a new type of quite big special numbers system, which is called third-order Jacobsthal 3-parameter generalized quaternions (shortly, third-order Jacobsthal 3PGQs). Further, we compute both some new equations and classical well-known equations such as: linear recurrence, Binet formulas, generating function, sum formulas, Cassini identity and d'Ocagne identity.

\end{abstract}

\medskip
\noindent

\medskip
\noindent
\medskip

\hrule

\section{Introduction}
Third-order Jacobsthal numbers and modified third-order Jacobsthal numbers are well-known numbers among integer sequences. Third-order Jacobsthal numbers satisfy the third order recurrence relation
$$
J_{0}^{(3)}=0,\ J_{1}^{(3)}=J_{2}^{(3)}=1,\ J_{n}^{(3)}=J_{n-1}^{(3)}+J_{n-2}^{(3)}+2J_{n-3}^{(3)},\ n\geq 3.
$$
Similarly, modified third-order Jacobsthal numbers are defined the following recurrence relation
$$
K_{0}^{(3)}=0,\ K_{1}^{(3)}=1,\ K_{2}^{(3)}=3,\ K_{n}^{(3)}=K_{n-1}^{(3)}+K_{n-2}^{(3)}+2K_{n-3}^{(3)},\ n\geq 3.
$$

Generating functions for the sequences $\{J_{n}^{(3)}\}_{n\geq0}$ and $\{K_{n}^{(3)}\}_{n\geq0}$ are
$$
J(x)=\sum_{n=0}^{\infty}J_{n}^{(3)}x^{n}=\frac{x}{1-x-x^{2}-2x^{3}}
$$
and
$$
K(x)=\sum_{n=0}^{\infty}K_{n}^{(3)}x^{n}=\frac{3-2x-x^{2}}{1-x-x^{2}-2x^{3}},
$$
respectively. Binet formulae for third-order Jacobsthal numbers and modified third-order Jacobsthal numbers are
$$
J_{n}^{(3)}=\frac{1}{7}\left[2^{n+1}+X_{n}-2X_{n+1}\right]
$$
and
$$
K_{n}^{(3)}=2^{n}+X_{n}+2X_{n+1},
$$
respectively, where $\{X_{n}\}_{n\geq 0}$ is the cyclic sequence defined by
$$
X_{n}=\frac{\omega_{1}^{n}-\omega_{2}^{n}}{\omega_{1}-\omega_{2}}=\left\{ 
\begin{array}{ccc}
0 & \textrm{if}& \mymod{n}{0}{3} \\ 
1& \textrm{if} & \mymod{n}{1}{3} \\ 
-1& \textrm{if} & \mymod{n}{2}{3}
\end{array}
\right. 
$$
and $\omega_{1}=\omega_{2}^{2}=\frac{1}{2}(-1+i\sqrt{3})\in \mathbb{C}$ are the roots of the equation $x^{2}+x+1=0$. For more details, the reader may review \cite{Coo,Mor1}.

On the other hand, \c{S}ent\"{u}rk and \"{U}nal \cite{Sen} introduced a generalized quaternions algebra called 3-parameter generalized quaternions (shortly, 3PGQs). The set of 3PGQs is denoted by $\mathbb{H}_{\lambda_{1},\lambda_{2},\lambda_{3}}$ and defined as
$$
\{\psi_{0}+\sum_{l=1}^{3}\psi_{l}\textbf{e}_{l}:\ \textbf{e}_{1}^{2}=-\lambda_{1}\lambda_{2},\ \textbf{e}_{2}^{2}=-\lambda_{1}\lambda_{3},\ \textbf{e}_{3}^{2}=-\lambda_{2}\lambda_{3},\ \textbf{e}_{1}\textbf{e}_{2}\textbf{e}_{3}=-\lambda_{1}\lambda_{2}\lambda_{3}\},
$$
where $\lambda_{l\in\{1,2,3\}}\in \mathbb{R}$.

The vectors $\{\textbf{e}_{1},\textbf{e}_{2},\textbf{e}_{3}\}$ satisfy the multiplication rules
\begin{equation}\label{e1}
\textbf{e}_{1}\textbf{e}_{2}=-\textbf{e}_{2}\textbf{e}_{1}=\lambda_{1}\textbf{e}_{3},\ \ \textbf{e}_{1}\textbf{e}_{3}=-\textbf{e}_{3}\textbf{e}_{1}=-\lambda_{2}\textbf{e}_{2},\ \ \textbf{e}_{2}\textbf{e}_{3}=-\textbf{e}_{3}\textbf{e}_{2}=\lambda_{3}\textbf{e}_{1}.
\end{equation}

Table \ref{t1} includes some special cases of 3PGQs. One can observe that bringing together the various types of quaternions and special recurrence sequence components is quite an attractive concept for several researchers in the literature. In \cite{Hal}, Hal\i c\i\ and Karata\c{s} studied the generalized Fibonacci quaternions, Morales \cite{Mor2} introduced generalized Tribonacci quaternions and the third-order Jacobsthal generalized quaternions are presented in \cite{Mor3}, Akyi\u{g}it et al. \cite{Aky1,Aky2} examined the split Fibonacci quaternions and generalized Fibonacci quaternions, and Br\'od introduced the split Horadam quaternions. Recently, K\i z\i late\c{s} and Kibar \cite{Kiz} determined the 3PGQs with higher order generalized Fibonacci numbers components, and \.{I}\c{s}bilir and G\"urses \cite{Isb} introduced the 3PGQs with Horadam numbers components.

\begin{table}[!htb]
\begin{tabular}{cccr}
\hline
$\lambda_{1} $ &$\lambda_{2}$ & $\lambda_{3}$ & Type of 3PGQs \\ [0.5ex] 
\hline 
$1$ & $\lambda_{2}$ & $\lambda_{3}$ & Generalized quaternions \\
$1$ & $1$ & $-1$ & Split quaternions \\
$1$ & $1$ & $1$ & Hamilton quaternions \\
$1$ & $1$ & $0$ & Semi-quaternions \\
$1$ & $-1$ & $0$ & Split semi-quaternions \\
\hline
\end{tabular}
\label{t1} 
\caption{Some types of 3PGQs.}
\end{table}

In this study, we intend to combine the third-order Jacobsthal numbers and 3PGQ; namely, we investigate 3PGQ with third-order Jacobsthal numbers components. Also, we examine the recurrence relation, Binet formula, generating function, summing formulas, d'Ocagne identity, Cassini identity, and some special equations as well.

\section{Third-order Jacobsthal 3-parameter generalized quaternions}
We define third-order Jacobsthal 3-parameter (TJ3p) generalized quaternions and modified third-order Jacobsthal 3-parameter (MTJ3p) generalized quaternions in following.

\begin{definition}\label{d}
For $n\geq 0$, the $n$-th TJ3p and MTJ3p generalized quaternions are
$$
J\mathcal{G}_{n}^{(3)}=J_{n}^{(3)}+J_{n+1}^{(3)}\textbf{e}_{1}+J_{n+2}^{(3)}\textbf{e}_{2}+J_{n+3}^{(3)}\textbf{e}_{3}
$$
and
$$
K\mathcal{G}_{n}^{(3)}=K_{n}^{(3)}+K_{n+1}^{(3)}\textbf{e}_{1}+K_{n+2}^{(3)}\textbf{e}_{2}+K_{n+3}^{(3)}\textbf{e}_{3},
$$
where $J_{n}^{(3)}$ and $K_{n}^{(3)}$ are the classical third-order Jacobsthal and modified third-order Jacobsthal numbers, $\{\textbf{e}_{1},\textbf{e}_{2},\textbf{e}_{3}\}$ satisfy the multiplication rules in Eq. (\ref{e1}).
\end{definition}

For $n\geq 3$, TJ3p and MTJ3p generalized quaternions satisfy the following recurrence relations
\begin{equation}\label{r1}
J\mathcal{G}_{n}^{(3)}=J\mathcal{G}_{n-1}^{(3)}+J\mathcal{G}_{n-2}^{(3)}+2J\mathcal{G}_{n-3}^{(3)}
\end{equation}
and
\begin{equation}\label{r2}
K\mathcal{G}_{n}^{(3)}=K\mathcal{G}_{n-1}^{(3)}+K\mathcal{G}_{n-2}^{(3)}+2K\mathcal{G}_{n-3}^{(3)},
\end{equation}
respectively. 

The next theorem gives generating functions for the sequences $\{J\mathcal{G}_{n}^{(3)}\}_{n\geq 0}$ and $\{K\mathcal{G}_{n}^{(3)}\}_{n\geq 0}$.  Since the proofs are very straightforward and easy, we do not prefer to give them.

\begin{theorem}\label{te1}
Generating functions for the sequences $\{J\mathcal{G}_{n}^{(3)}\}_{n\geq 0}$ and $\{K\mathcal{G}_{n}^{(3)}\}_{n\geq 0}$ are
$$
\sum_{n=0}^{\infty}J\mathcal{G}_{n}^{(3)}x^{n}=\frac{1}{\sigma(x)}\left[\textnormal{\textbf{e}}_{1}+\textnormal{\textbf{e}}_{2}+2\textnormal{\textbf{e}}_{3}+(1+\textnormal{\textbf{e}}_{2}+3\textnormal{\textbf{e}}_{3})x+2(\textnormal{\textbf{e}}_{2}+\textnormal{\textbf{e}}_{3})x^{2}\right]
$$
and
$$
\sum_{n=0}^{\infty}K\mathcal{G}_{n}^{(3)}x^{n}=\frac{1}{\sigma(x)}\left\lbrace \begin{array}{c}
3+\textnormal{\textbf{e}}_{1}+3\textnormal{\textbf{e}}_{2}+10\textnormal{\textbf{e}}_{3}-(2-2\textnormal{\textbf{e}}_{1}-7\textnormal{\textbf{e}}_{2}-5\textnormal{\textbf{e}}_{3})x\\
-(1-6\textnormal{\textbf{e}}_{1}-2\textnormal{\textbf{e}}_{2}-6\textnormal{\textbf{e}}_{3})x^{2}
\end{array}\right\rbrace,
$$
where $\sigma(x)=1-x-x^{2}-2x^{3}$.
\end{theorem}

An interesting property of this type of quaternions is presented in the following theorem. To prove this result, we will use the relations that satisfy the third-order Jacobsthal and modified third-order Jacobsthal numbers.
\begin{theorem}\label{te2}
For any non-negative integer $n$, we have 
$$
J\mathcal{G}_{n+3}^{(3)}=J\mathcal{G}_{n}^{(3)}+2^{n+1}\Theta
$$
and
$$
K\mathcal{G}_{n+3}^{(3)}=K\mathcal{G}_{n}^{(3)}+7\cdot2^{n}\Theta
$$
where $\Theta=1+2\textnormal{\textbf{e}}_{1}+4\textnormal{\textbf{e}}_{2}+8\textnormal{\textbf{e}}_{3}$.
\end{theorem}
\begin{proof}
Considering the Definition \ref{d} and the relation $J_{n+3}^{(3)}-2^{n+1}=J_{n}^{(3)}$ for $n\geq 0$, we get
\begin{align*}
J\mathcal{G}_{n}^{(3)}&=J_{n}^{(3)}+J_{n+1}^{(3)}\textbf{e}_{1}+J_{n+2}^{(3)}\textbf{e}_{2}+J_{n+3}^{(3)}\textbf{e}_{3}\\
&=J_{n+3}^{(3)}-2^{n+1}+\left[J_{n+4}^{(3)}-2^{n+2}\right]\textbf{e}_{1}+\left[J_{n+5}^{(3)}-2^{n+3}\right]\textbf{e}_{2}+\left[J_{n+6}^{(3)}-2^{n+4}\right]\textbf{e}_{3}\\
&=J_{n+3}^{(3)}+J_{n+4}^{(3)}\textbf{e}_{1}+J_{n+5}^{(3)}\textbf{e}_{2}+J_{n+6}^{(3)}\textbf{e}_{3}-2^{n+1}(1+2\textnormal{\textbf{e}}_{1}+4\textnormal{\textbf{e}}_{2}+8\textnormal{\textbf{e}}_{3})\\
&=J\mathcal{G}_{n+3}^{(3)}-2^{n+1}\Theta.
\end{align*} 
The last equation gives the first equation in theorem. The proof of the second equation is very similar using the identity $K_{n+3}^{(3)}-7\cdot 2^{n}=K_{n}^{(3)}$ for modified third-order Jacobsthal numbers.
\end{proof}

Binet formulas for the TJ3p and MTJ3p generalized quaternions can be found in the following theorem.
\begin{theorem}\label{te3}
For $n\geq 0$, the $n$-th TJ3p and MTJ3p generalized quaternions are, respectively
\begin{align*}
J\mathcal{G}_{n}^{(3)}&=\frac{1}{7}\left[2^{n+1}\Theta+X_{n}\textnormal{\textbf{A}}-X_{n+1}\textnormal{\textbf{B}}\right]\\
&=\frac{1}{7}\left\{ 
\begin{array}{ccc}
2^{n+1}\Theta-\textnormal{\textbf{B}}& \textrm{if}& \mymod{n}{0}{3} \\ 
2^{n+1}\Theta+\textnormal{\textbf{A}}+\textnormal{\textbf{B}}& \textrm{if} & \mymod{n}{1}{3} \\ 
2^{n+1}\Theta-\textnormal{\textbf{A}}& \textrm{if} & \mymod{n}{2}{3}
\end{array}
\right. 
\end{align*}
and
\begin{align*}
K\mathcal{G}_{n}^{(3)}&=2^{n}\Theta+X_{n}\textnormal{\textbf{C}}-X_{n+1}\textnormal{\textbf{D}}=\left\{ 
\begin{array}{ccc}
2^{n}\Theta-\textnormal{\textbf{D}}& \textrm{if}& \mymod{n}{0}{3} \\ 
2^{n}\Theta+\textnormal{\textbf{C}}+\textnormal{\textbf{D}}& \textrm{if} & \mymod{n}{1}{3} \\ 
2^{n}\Theta-\textnormal{\textbf{C}}& \textrm{if} & \mymod{n}{2}{3}
\end{array}
\right. ,
\end{align*}
where $\Theta=1+2\textnormal{\textbf{e}}_{1}+4\textnormal{\textbf{e}}_{2}+8\textnormal{\textbf{e}}_{3}$, $\textnormal{\textbf{A}}=1+2\textnormal{\textbf{e}}_{1}-3\textnormal{\textbf{e}}_{2}+\textnormal{\textbf{e}}_{3}$,  $\textnormal{\textbf{B}}=2-3\textnormal{\textbf{e}}_{1}+\textnormal{\textbf{e}}_{2}+2\textnormal{\textbf{e}}_{3}$, $\textnormal{\textbf{C}}=1-2\textnormal{\textbf{e}}_{1}+\textnormal{\textbf{e}}_{2}+\textnormal{\textbf{e}}_{3}$ and $\textnormal{\textbf{D}}=-2+\textnormal{\textbf{e}}_{1}+\textnormal{\textbf{e}}_{2}-2\textnormal{\textbf{e}}_{3}$.
\end{theorem}
\begin{proof}
From the Definition \ref{d} for TJ3p generalized quaternions, the Binet formula for the third-order Jacobsthal numbers and recurrence relation $X_{n+2}=-X_{n+1}-X_{n}$ for $n\geq 0$, we have
\begin{align*}
7J\mathcal{G}_{n}^{(3)}&=7J_{n}^{(3)}+7J_{n+1}^{(3)}\textbf{e}_{1}+7J_{n+2}^{(3)}\textbf{e}_{2}+7J_{n+3}^{(3)}\textbf{e}_{3}\\
&=2^{n+1}+X_{n}-2X_{n+1}\\
&\ \ + \left[2^{n+2}+X_{n+1}-2X_{n+2}\right]\textbf{e}_{1}\\
&\ \ + \left[2^{n+3}+X_{n+2}-2X_{n+3}\right]\textbf{e}_{2}\\
&\ \ + \left[2^{n+4}+X_{n+3}-2X_{n+4}\right]\textbf{e}_{3}\\
&=2^{n+1}(1+2\textbf{e}_{1}+4\textbf{e}_{2}+8\textbf{e}_{3})+X_{n}(1+2\textbf{e}_{1}-3\textbf{e}_{2}+\textbf{e}_{3})\\
&\ \ - X_{n+1}(2-3\textbf{e}_{1}+\textbf{e}_{2}+2\textbf{e}_{3})\\
&=2^{n+1}\Theta+X_{n}\textnormal{\textbf{A}}-X_{n+1}\textnormal{\textbf{B}},
\end{align*}
where $\Theta=1+2\textbf{e}_{1}+4\textbf{e}_{2}+8\textbf{e}_{3}$, $\textnormal{\textbf{A}}=1+2\textbf{e}_{1}-3\textbf{e}_{2}+\textbf{e}_{3}$ and $\textnormal{\textbf{B}}=2-3\textbf{e}_{1}+\textbf{e}_{2}+2\textbf{e}_{3}$. The last equation gives the first equation in theorem. The proof of the second equation is very similar.
\end{proof}
The notation $\Theta=1+2\textnormal{\textbf{e}}_{1}+4\textnormal{\textbf{e}}_{2}+8\textnormal{\textbf{e}}_{3}$, $\textnormal{\textbf{A}}=1+2\textnormal{\textbf{e}}_{1}-3\textnormal{\textbf{e}}_{2}+\textnormal{\textbf{e}}_{3}$,  $\textnormal{\textbf{B}}=2-3\textnormal{\textbf{e}}_{1}+\textnormal{\textbf{e}}_{2}+2\textnormal{\textbf{e}}_{3}$, $\textnormal{\textbf{C}}=1-2\textnormal{\textbf{e}}_{1}+\textnormal{\textbf{e}}_{2}+\textnormal{\textbf{e}}_{3}$ and $\textnormal{\textbf{D}}=-2+\textnormal{\textbf{e}}_{1}+\textnormal{\textbf{e}}_{2}-2\textnormal{\textbf{e}}_{3}$ of the previous theorem will be used in the rest of the paper.

It should be noted that
\begin{align*}
\textnormal{\textbf{B}}\textnormal{\textbf{A}}&=(2-3\textbf{e}_{1}+\textbf{e}_{2}+2\textbf{e}_{3})(1+2\textbf{e}_{1}-3\textbf{e}_{2}+\textbf{e}_{3})\\
&=2+4\textbf{e}_{1}-6\textbf{e}_{2}+2\textbf{e}_{3}\\
&\ \ -3\textbf{e}_{1}+6\lambda_{1}\lambda_{2}+9\textbf{e}_{3}+3\lambda_{2}\textbf{e}_{2}\\
&\ \ +\textbf{e}_{2}-2\lambda_{1}\textbf{e}_{3}+3\lambda_{1}\lambda_{3}+\lambda_{3}\textbf{e}_{1}\\
&\ \ +2\textbf{e}_{3}+4\lambda_{2}\textbf{e}_{2}+6\lambda_{3}\textbf{e}_{1}-2\lambda_{2}\lambda_{3}\\
&=2+6\lambda_{1}\lambda_{2}+3\lambda_{1}\lambda_{3}-2\lambda_{2}\lambda_{3}+\textbf{e}_{1}-5\textbf{e}_{2}+4\textbf{e}_{3}+
7(\lambda_{3}\textbf{e}_{1}+\lambda_{2}\textbf{e}_{2}+\lambda_{1}\textbf{e}_{3})\\
&=\Psi+7\Omega,
\end{align*}
where $\Psi=2+6\lambda_{1}\lambda_{2}+3\lambda_{1}\lambda_{3}-2\lambda_{2}\lambda_{3}+\textbf{e}_{1}-5\textbf{e}_{2}+4\textbf{e}_{3}$ and $\Omega=\lambda_{3}\textbf{e}_{1}+\lambda_{2}\textbf{e}_{2}+\lambda_{1}\textbf{e}_{3}$.

Then, using the multiplicative rule in Eq. (\ref{e1}), we obtain the next identities
\begin{equation}\label{m1}
\textnormal{\textbf{B}}\textnormal{\textbf{A}}=\Psi+7\Omega,\ \ \textnormal{\textbf{A}}\textnormal{\textbf{B}}=\Psi-7\Omega
\end{equation}
and
\begin{equation}\label{m2}
\textnormal{\textbf{D}}\textnormal{\textbf{C}}=\Phi+3\Omega,\ \ \textnormal{\textbf{C}}\textnormal{\textbf{D}}=\Phi-3\Omega,
\end{equation}
where
\begin{align*}
\Psi&=2+6\lambda_{1}\lambda_{2}+3\lambda_{1}\lambda_{3}-2\lambda_{2}\lambda_{3}+\textbf{e}_{1}-5\textbf{e}_{2}+4\textbf{e}_{3},\\
\Phi&=-2+2\lambda_{1}\lambda_{2}-\lambda_{1}\lambda_{3}+2\lambda_{2}\lambda_{3}+5\textbf{e}_{1}-\textbf{e}_{2}-4\textbf{e}_{3},\\
\Omega&=\lambda_{3}\textbf{e}_{1}+\lambda_{2}\textbf{e}_{2}+\lambda_{1}\textbf{e}_{3}.
\end{align*}

Eqs. (\ref{m1}) and (\ref{m2}) will be used for simplification in the next section. It should be noted that
\begin{equation}\label{m3}
\textnormal{\textbf{B}}\textnormal{\textbf{A}}=\textnormal{\textbf{A}}\textnormal{\textbf{B}}+14\Omega,\ \ \textnormal{\textbf{D}}\textnormal{\textbf{C}}=\textnormal{\textbf{C}}\textnormal{\textbf{D}}+6\Omega.
\end{equation}

\section{Main Results}
In this section, we present generalizations for the some well-known identities of third-order Jacobsthal and modified third-order Jacobsthal numbers. We begin with the Vajda's identities for the TJ3p and MTJ3p generalized quaternions given in the following theorem.

\begin{theorem}\label{te4}
For integers $n$, $a$ and $b$, we have
$$
J\mathcal{G}_{n+a}^{(3)}J\mathcal{G}_{n+b}^{(3)}-J\mathcal{G}_{n}^{(3)}J\mathcal{G}_{n+a+b}^{(3)}=\frac{1}{49}\left\lbrace \begin{array}{c}
2^{n+1}\left[\Theta Y_{n+b}(a)-2^{b}Y_{n}(a)\Theta\right]\\
+X_{a}\left[X_{b}\varXi-14X_{b+2}\Omega\right]
\end{array}\right\rbrace 
$$
and
$$
K\mathcal{G}_{n+a}^{(3)}K\mathcal{G}_{n+b}^{(3)}-K\mathcal{G}_{n}^{(3)}K\mathcal{G}_{n+a+b}^{(3)}=\left\lbrace \begin{array}{c}
2^{n}\left[\Theta Y_{n+b}^{*}(a)-2^{b}Y_{n}^{*}(a)\Theta\right]\\
+X_{a}\left[X_{b}\varXi^{*}-6X_{b+2}\Omega\right]
\end{array}\right\rbrace ,
$$
where $Y_{n}(a)=X_{n}(2^{a}\textnormal{\textbf{A}}+X_{a+2}\textnormal{\textbf{A}}-X_{a}\textnormal{\textbf{B}})-X_{n+1}(2^{a}\textnormal{\textbf{B}}+X_{a}\textnormal{\textbf{A}}-X_{a+1}\textnormal{\textbf{B}})$, $Y_{n}^{*}(a)=X_{n}(2^{a}\textnormal{\textbf{C}}+X_{a+2}\textnormal{\textbf{C}}-X_{a}\textnormal{\textbf{D}})-X_{n+1}(2^{a}\textnormal{\textbf{D}}+X_{a}\textnormal{\textbf{C}}-X_{a+1}\textnormal{\textbf{D}})$, $\varXi=\textnormal{\textbf{A}}^{2}+\textnormal{\textbf{A}}\textnormal{\textbf{B}}+\textnormal{\textbf{B}}^{2}$ and $\varXi^{*}=\textnormal{\textbf{C}}^{2}+\textnormal{\textbf{C}}\textnormal{\textbf{D}}+\textnormal{\textbf{D}}^{2}$.
\end{theorem}
\begin{proof}
Let us consider $Z_{n}=X_{n}\textnormal{\textbf{A}}-X_{n+1}\textnormal{\textbf{B}}$. Using the Vajda's identity for sequence $\{X_{n}\}$, $X_{n+a}X_{n+b}-X_{n}X_{n+a+b}=X_{a}X_{b}$, we have
\begin{align*}
Z_{n+a}Z_{n+b}&-Z_{n}Z_{n+a+b}\\
&=\left[X_{n+a}\textnormal{\textbf{A}}-X_{n+a+1}\textnormal{\textbf{B}}\right]\left[X_{n+b}\textnormal{\textbf{A}}-X_{n+b+1}\textnormal{\textbf{B}}\right]\\
&\ \ - \left[X_{n}\textnormal{\textbf{A}}-X_{n+1}\textnormal{\textbf{B}}\right]\left[X_{n+a+b}\textnormal{\textbf{A}}-X_{n+a+b+1}\textnormal{\textbf{B}}\right]\\
&=\left[X_{n+a}X_{n+b}-X_{n}X_{n+a+b}\right]\textnormal{\textbf{A}}^{2}\\
&\ \ - \left[X_{n+a}X_{n+b+1}-X_{n}X_{n+a+b+1}\right]\textnormal{\textbf{A}}\textnormal{\textbf{B}}\\
&\ \ - \left[X_{n+a+1}X_{n+b}-X_{n+1}X_{n+a+b}\right]\textnormal{\textbf{B}}\textnormal{\textbf{A}}\\
&\ \ + \left[X_{n+a+1}X_{n+b+1}-X_{n+1}X_{n+a+b+1}\right]\textnormal{\textbf{B}}^{2}\\
&=X_{a}X_{b}\textnormal{\textbf{A}}^{2}-X_{a}X_{b+1}\textnormal{\textbf{A}}\textnormal{\textbf{B}}-X_{a}X_{b-1}\textnormal{\textbf{B}}\textnormal{\textbf{A}}+X_{a}X_{b}\textnormal{\textbf{B}}^{2}\\
&=X_{a}X_{b}\textnormal{\textbf{A}}^{2}-X_{a}X_{b+1}\textnormal{\textbf{A}}\textnormal{\textbf{B}}-X_{a}X_{b-1}(\textnormal{\textbf{A}}\textnormal{\textbf{B}}+14\Omega)+X_{a}X_{b}\textnormal{\textbf{B}}^{2}\\
&=X_{a}X_{b}(\textnormal{\textbf{A}}^{2}+\textnormal{\textbf{A}}\textnormal{\textbf{B}}+\textnormal{\textbf{B}}^{2})-14X_{a}X_{b+2}\Omega\\
&=X_{a}\left[X_{b}\varXi -14X_{b+2}\Omega\right],
\end{align*}
where $\varXi=\textnormal{\textbf{A}}^{2}+\textnormal{\textbf{A}}\textnormal{\textbf{B}}+\textnormal{\textbf{B}}^{2}$. From the Binet formula for the TJ3p generalized quaternions in Theorem \ref{te3}, we obtain
\begin{align*}
49\left[J\mathcal{G}_{n+a}^{(3)}J\mathcal{G}_{n+b}^{(3)}-J\mathcal{G}_{n}^{(3)}J\mathcal{G}_{n+a+b}^{(3)}\right]&=\left[2^{n+a+1}\Theta+Z_{n+a}\right]\left[2^{n+b+1}\Theta+Z_{n+b}\right]\\
&\ \ - \left[2^{n+1}\Theta+Z_{n}\right]\left[2^{n+a+b+1}\Theta+Z_{n+a+b}\right]\\
&=2^{n+a+1}\Theta Z_{n+b}+2^{n+b+1}Z_{n+a}\Theta \\
&\ \ - 2^{n+1}\Theta Z_{n+a+b}+2^{n+a+b+1}Z_{n}\Theta \\
&\ \ + Z_{n+a}Z_{n+b}-Z_{n}Z_{n+a+b}\\
&=2^{n+1}\left[\Theta Y_{n+b}(a)-2^{b}Y_{n}(a)\Theta\right]\\
&\ \ +X_{a}\left[X_{b}\varXi-14X_{b+2}\Omega\right],
\end{align*}
where $Y_{n}(a)=X_{n}(2^{a}\textnormal{\textbf{A}}+X_{a+2}\textnormal{\textbf{A}}-X_{a}\textnormal{\textbf{B}})-X_{n+1}(2^{a}\textnormal{\textbf{B}}+X_{a}\textnormal{\textbf{A}}-X_{a+1}\textnormal{\textbf{B}})$. The last equation gives the first equation in theorem. The proof of the second equation is very similar.
\end{proof}

 Taking $b=-a$ in the Theorem \ref{te4} with the identities $X_{-n}=-X_{n}$ and $X_{n}=X_{n+3}$ gives Catalan's identities for the TJ3p and MTJ3p generalized quaternions.

\begin{corollary}\label{co1}
Let $n$ and $a$ be two integers such that $n\geq a$, we have
$$
J\mathcal{G}_{n+a}^{(3)}J\mathcal{G}_{n-a}^{(3)}-\left[J\mathcal{G}_{n}^{(3)}\right]^{2}=\frac{1}{49}\left\lbrace \begin{array}{c}
2^{n+1}\left[\Theta Y_{n-a}(a)-2^{-a}Y_{n}(a)\Theta\right]\\
-X_{a}\left[X_{a}\varXi-14X_{a+1}\Omega\right]
\end{array}\right\rbrace 
$$
and
$$
K\mathcal{G}_{n+a}^{(3)}K\mathcal{G}_{n-a}^{(3)}-\left[K\mathcal{G}_{n}^{(3)}\right]^{2}=\left\lbrace \begin{array}{c}
2^{n}\left[\Theta Y_{n-a}^{*}(a)-2^{-a}Y_{n}^{*}(a)\Theta\right]\\
-X_{a}\left[X_{a}\varXi^{*}-6X_{a+1}\Omega\right]
\end{array}\right\rbrace ,
$$
with $Y_{n}(a)$, $Y_{n}^{*}(a)$, $\varXi$ and $\varXi^{*}$ as in Theorem \ref{te4}.
\end{corollary}

Similarly, taking $a=1$ in Corollary \ref{co1} gives Cassini's identities for the TJ3p and MTJ3p generalized quaternions.
\begin{corollary}\label{co2}
For any integer $n\geq 1$, we have
$$
J\mathcal{G}_{n+1}^{(3)}J\mathcal{G}_{n-1}^{(3)}-\left[J\mathcal{G}_{n}^{(3)}\right]^{2}=\frac{1}{49}\left[
2^{n}\left(2\Theta Y_{n-1}(1)-Y_{n}(1)\Theta\right)-\left(\varXi+14\Omega\right)\right] 
$$
and
$$
K\mathcal{G}_{n+1}^{(3)}K\mathcal{G}_{n-1}^{(3)}-\left[K\mathcal{G}_{n}^{(3)}\right]^{2}=2^{n-1}\left[2\Theta Y_{n-1}^{*}(1)-Y_{n}^{*}(1)\Theta\right]-\left[\varXi^{*}+6\Omega\right],
$$
with $Y_{n}(1)=X_{n}(2\textnormal{\textbf{A}}-\textnormal{\textbf{B}})-X_{n+1}(3\textnormal{\textbf{B}}+\textnormal{\textbf{A}})$, $Y_{n}^{*}(1)=X_{n}(2\textnormal{\textbf{C}}-\textnormal{\textbf{D}})-X_{n+1}(3\textnormal{\textbf{D}}+\textnormal{\textbf{C}})$, $\varXi$ and $\varXi^{*}$ as in Theorem \ref{te4}.
\end{corollary}

For example, if $n=1$ in the previous corollary, we get
\begin{align*}
J\mathcal{G}_{0}^{(3)}&=\frac{1}{7}\left[2\Theta-\textnormal{\textbf{B}}\right],\\
J\mathcal{G}_{1}^{(3)}&=\frac{1}{7}\left[4\Theta+\textnormal{\textbf{A}}+\textnormal{\textbf{B}}\right],\\
J\mathcal{G}_{2}^{(3)}&=\frac{1}{7}\left[8\Theta-\textnormal{\textbf{A}}\right].
\end{align*}

Then, we can write
\begin{align*}
49\left[J\mathcal{G}_{2}^{(3)}J\mathcal{G}_{0}^{(3)}-\left(J\mathcal{G}_{1}^{(3)}\right)^{2}\right]&=\left[8\Theta-\textnormal{\textbf{A}}\right]\left[2\Theta-\textnormal{\textbf{B}}\right]-\left[4\Theta+\textnormal{\textbf{A}}+\textnormal{\textbf{B}}\right]^{2}\\
&=-4\Theta (\textnormal{\textbf{A}}+3\textnormal{\textbf{B}})-2(3\textnormal{\textbf{A}}+2\textnormal{\textbf{B}})\Theta - (\textnormal{\textbf{A}}^{2}+\textnormal{\textbf{B}}\textnormal{\textbf{A}}+\textnormal{\textbf{B}}^{2})
\end{align*}
and
$$
J\mathcal{G}_{2}^{(3)}J\mathcal{G}_{0}^{(3)}-\left[J\mathcal{G}_{1}^{(3)}\right]^{2}=\frac{1}{49}\left[
2\left(2\Theta Y_{0}(1)-Y_{1}(1)\Theta\right)-\left(\varXi+14\Omega\right)\right],
$$
where $Y_{0}(1)=-(\textnormal{\textbf{A}}+3\textnormal{\textbf{B}})$ and $Y_{1}(1)=3\textnormal{\textbf{A}}+2\textnormal{\textbf{B}}$.

Following identities can be proved by using definitions, recurrence relations and Binet formulas for the TJ3p and MTJ3p generalized quaternions.

\begin{theorem}\label{te6}
For any integer $n\geq 0$, we have 
\begin{align*}
J\mathcal{G}_{n}^{(3)}+J\mathcal{G}_{n+1}^{(3)}+J\mathcal{G}_{n+2}^{(3)}&=2^{n+1}\Theta,\\
K\mathcal{G}_{n}^{(3)}+K\mathcal{G}_{n+1}^{(3)}+K\mathcal{G}_{n+2}^{(3)}&=7\cdot2^{n}\Theta,\\
K\mathcal{G}_{n+2}^{(3)}&=J\mathcal{G}_{n+2}^{(3)}+2J\mathcal{G}_{n+1}^{(3)}+6J\mathcal{G}_{n}^{(3)},\\
J\mathcal{G}_{n+2}^{(3)}&=\frac{1}{147}\left[13K\mathcal{G}_{n+2}^{(3)}+48K\mathcal{G}_{n+1}^{(3)}+20K\mathcal{G}_{n}^{(3)}\right],
\end{align*}
where $\Theta=1+2\textnormal{\textbf{e}}_{1}+4\textnormal{\textbf{e}}_{2}+8\textnormal{\textbf{e}}_{3}$.
\end{theorem}

\section{Sums of the TJ3p and MTJ3p generalized quaternions}
In this section, we present some results concerning sums of terms of the TJ3p and MTJ3p generalized quaternions. Some known equations about this numbers will be given in the following lemma.

\begin{lemma}
For any integer $n\geq 0$, we have
\begin{equation}\label{s1}
\sum_{l=0}^{n}J_{l}^{(3)}=\frac{1}{3}\left[J_{n+2}^{(3)}+2J_{n}^{(3)}-1\right]
\end{equation}
and
\begin{equation}\label{s2}
\sum_{l=0}^{n}K_{l}^{(3)}=\frac{1}{3}\left[K_{n+2}^{(3)}+2K_{n}^{(3)}\right].
\end{equation}
\end{lemma}
\begin{proof}
The proof is easy using induction on $n$.
\end{proof}

The next theorem presents a summation formula for the third-order Jacobsthal and modified third-order Jacobsthal 3-parameter generalized quaternions.

\begin{theorem}\label{te7}
Let $n\geq 0$ be an integer. Then
$$
\sum_{l=0}^{n}J\mathcal{G}_{l}^{(3)}=\frac{1}{3}\left[J\mathcal{G}_{n+2}^{(3)}+2J\mathcal{G}_{n}^{(3)}-\left(1+\textnormal{\textbf{e}}_{1}+4\textnormal{\textbf{e}}_{2}+7\textnormal{\textbf{e}}_{3}\right)\right]
$$
and
$$
\sum_{l=0}^{n}K\mathcal{G}_{l}^{(3)}=\frac{1}{3}\left[K\mathcal{G}_{n+2}^{(3)}+2K\mathcal{G}_{n}^{(3)}-\left(9\textnormal{\textbf{e}}_{1}+12\textnormal{\textbf{e}}_{2}+21\textnormal{\textbf{e}}_{3}\right)\right].
$$
\end{theorem}
\begin{proof}
Using Definition \ref{d} and Eq. (\ref{s1}), we get
\begin{align*}
3\sum_{l=0}^{n}J\mathcal{G}_{l}^{(3)}&=3\sum_{l=0}^{n}\left[J_{l}^{(3)}+J_{l+1}^{(3)}\textbf{e}_{1}+J_{l+2}^{(3)}\textbf{e}_{2}+J_{l+3}^{(3)}\textbf{e}_{3}\right]\\
&=3\sum_{l=0}^{n}J_{l}^{(3)}+3\left(\sum_{l=0}^{n}J_{l+1}^{(3)}\right)\textbf{e}_{1}+3\left(\sum_{l=0}^{n}J_{l+2}^{(3)}\right)\textbf{e}_{2}+3\left(\sum_{l=0}^{n}J_{l+3}^{(3)}\right)\textbf{e}_{3}\\
&=J_{n+2}^{(3)}+2J_{n}^{(3)}-1+\left(J_{n+3}^{(3)}+2J_{n+1}^{(3)}-1\right)\textbf{e}_{1}\\
&\ \ + \left(J_{n+4}^{(3)}+2J_{n+2}^{(3)}-4\right)\textbf{e}_{2}+\left(J_{n+5}^{(3)}+2J_{n+3}^{(3)}-7\right)\textbf{e}_{3}\\
&=J_{n+2}^{(3)}+J_{n+3}^{(3)}\textbf{e}_{1}+J_{n+4}^{(3)}\textbf{e}_{2}+J_{n+5}^{(3)}\textbf{e}_{3}\\
&\ \ + 2\left(J_{n}^{(3)}+J_{n+1}^{(3)}\textbf{e}_{1}+J_{n+2}^{(3)}\textbf{e}_{2}+J_{n+3}^{(3)}\textbf{e}_{3}\right)\\
&\ \ - \left(1+\textbf{e}_{1}+4\textbf{e}_{2}+7\textbf{e}_{3}\right)\\
&=J\mathcal{G}_{n+2}^{(3)}+2J\mathcal{G}_{n}^{(3)}-\left(1+\textbf{e}_{1}+4\textbf{e}_{2}+7\textbf{e}_{3}\right).
\end{align*}
The other identity is similar using Definition \ref{d} and Eq. (\ref{s2}).
\end{proof}

\begin{theorem}
Let $n\geq 0$ be an integer. Then
\begin{equation}\label{ne1}
J\mathcal{G}_{n+2}^{(3)}-4J\mathcal{G}_{n}^{(3)}=-\frac{1}{7}\left[X_{n}(5\textnormal{\textbf{A}}+\textnormal{\textbf{B}})+X_{n+1}(\textnormal{\textbf{A}}-4\textnormal{\textbf{B}})\right],
\end{equation}
\begin{equation}\label{ne2}
K\mathcal{G}_{n+2}^{(3)}-4K\mathcal{G}_{n}^{(3)}=-X_{n}(5\textnormal{\textbf{C}}+\textnormal{\textbf{D}})-X_{n+1}(\textnormal{\textbf{C}}-4\textnormal{\textbf{D}}).
\end{equation}
\end{theorem}
\begin{proof}
By the Definition \ref{d} and Theorem \ref{te3}, we get
\begin{align*}
J\mathcal{G}_{n+2}^{(3)}-4J\mathcal{G}_{n}^{(3)}&=\frac{1}{7}\left[2^{n+3}\Theta+X_{n+2}\textnormal{\textbf{A}}-X_{n+3}\textnormal{\textbf{B}}\right]\\
&\ \ - \frac{4}{7}\left[2^{n+1}\Theta+X_{n}\textnormal{\textbf{A}}-X_{n+1}\textnormal{\textbf{B}}\right]\\
&=\frac{1}{7}\left[X_{n+2}\textnormal{\textbf{A}}-X_{n+3}\textnormal{\textbf{B}}-4X_{n}\textnormal{\textbf{A}}+4X_{n+1}\textnormal{\textbf{B}}\right]\\
&=\frac{1}{7}\left[-X_{n+1}\textnormal{\textbf{A}}-X_{n}\textnormal{\textbf{A}}-X_{n}\textnormal{\textbf{B}}-4X_{n}\textnormal{\textbf{A}}+4X_{n+1}\textnormal{\textbf{B}}\right]\\
&=-\frac{1}{7}\left[X_{n}(5\textnormal{\textbf{A}}+\textnormal{\textbf{B}})+X_{n+1}(\textnormal{\textbf{A}}-4\textnormal{\textbf{B}})\right],
\end{align*}
which completes the proof of (\ref{ne1}). The equation (\ref{ne2}) is similar using Theorem \ref{te3}.
\end{proof}

In the same manner we can prove next results.
\begin{theorem}
Let $n\geq 0$ be an integer. Then
\begin{equation}
7J\mathcal{G}_{n}^{(3)}+K\mathcal{G}_{n}^{(3)}=3\cdot 2^{n}\Theta+X_{n}(\textnormal{\textbf{A}}+\textnormal{\textbf{C}})-X_{n+1}(\textnormal{\textbf{B}}+\textnormal{\textbf{D}}),
\end{equation}
\begin{equation}
7J\mathcal{G}_{n}^{(3)}-K\mathcal{G}_{n}^{(3)}=2^{n}\Theta+X_{n}(\textnormal{\textbf{A}}-\textnormal{\textbf{C}})-X_{n+1}(\textnormal{\textbf{B}}-\textnormal{\textbf{D}}),
\end{equation}
where $\Theta=1+2\textnormal{\textbf{e}}_{1}+4\textnormal{\textbf{e}}_{2}+8\textnormal{\textbf{e}}_{3}$.
\end{theorem}

\section{Conclusions}
This study presents the 3-parameter generalized quaternions, the consecutive coefficient third-order Jacobsthal and modified third-order Jacobsthal numbers. We obtained these new numbers not defined in the literature before. We have given a comprehensive introductory study of 3-parameter generalized quaternions as a guide. Since this study includes some new results, it contributes to literature by providing essential information concerning these new numbers. Therefore, we hope that these new numbers and properties that we have found will offer a new perspective to the researches. As a future direction, we plan to study other identities of  generalized tribonacci 3-parameter generalized quaternion by increasing the diversity of these number sequences.

\medskip

\end{document}